\documentclass[11pt]{amsart}
\usepackage{amsmath,amssymb,latexsym,soul,cite,mathrsfs}

\usepackage{color,enumitem,graphicx}
\usepackage{dsfont}
\usepackage[colorlinks=true,urlcolor=blue,
citecolor=red,linkcolor=blue,linktocpage,pdfpagelabels,
bookmarksnumbered,bookmarksopen]{hyperref}
\usepackage[english]{babel}

\usepackage[left=2.9cm,right=2.9cm,top=2.8cm,bottom=2.8cm]{geometry}
\usepackage[hyperpageref]{backref}
\usepackage{graphicx}

\usepackage[colorinlistoftodos]{todonotes}
\makeatletter
\providecommand\@dotsep{5}
\def\listtodoname{List of Todos}
\def\listoftodos{\@starttoc{tdo}\listtodoname}
\makeatother

\numberwithin{equation}{section}

\def\R {{\rm I}\hskip -0.85mm{\rm R}}
\def\N {{\rm I}\hskip -0.85mm{\rm N}}

\DeclareMathOperator*{\essinf}{ess\,inf}

\newcommand{\rig}{\rightarrow}

\newtheorem{theorem}{Theorem}[section]
\newtheorem{proposition}[theorem]{Proposition}
\newtheorem{lemma}[theorem]{Lemma}

\newtheorem{remark}{Remark}

\title[Multiplicity of solutions for asymptotically linear elliptic problems]
{ Multiplicity of solutions for resonant and non-resonant asymptotically linear elliptic problems}

\author[J. R. S. Nascimento]{Jos\'e R. S. Nascimento}
\author[M. T. O. Pimenta]{Marcos T. O. Pimenta}
\author[J. R. Santos Jr.]{Jo\~ao R. Santos J\'unior}

\address[J.R.S. Nascimento]{\newline\indent Faculdade de Matem\'atica
\newline\indent 
Instituto de Ci\^{e}ncias Exatas e Naturais
\newline\indent 
Universidade Federal do Par\'a
\newline\indent
Avenida Augusto corr\^{e}a 01, 66075-110, Bel\'em, PA, Brazil}
\email{\href{mailto:  jrsn19@hotmail.com}{jrsn19@hotmail.com}}

\address[M. T. O. Pimenta]{\newline\indent Departamento de Matem\'atica e Computa\c{c}\~ao
\newline\indent 
Faculdade de Ci\^encias e Tecnologia
\newline\indent 
Universidade Estadual Paulista - UNESP
\newline\indent 
Rua Roberto Simonsen, 305, 19060-900, Presidente Prudente, SP, Brazil }
\email{\href{mailto: pimenta@fct.unesp.br}{pimenta@fct.unesp.br}}

\address[J. R. Santos Jr.]{\newline\indent Faculdade de Matem\'atica
\newline\indent 
Instituto de Ci\^{e}ncias Exatas e Naturais
\newline\indent 
Universidade Federal do Par\'a
\newline\indent
Avenida Augusto corr\^{e}a 01, 66075-110, Bel\'em, PA, Brazil}
\email{\href{mailto: joaojunior@ufpa.br }{joaojunior@ufpa.br}}

\thanks{J. R. S. Nascimento was partially
supported by CAPES/FAPESPA-Proc. 00.889.834/0001-08, Brazil. J. R. Santos J\'unior was partially
supported by CNPq-Proc. 302698/2015-9 and CAPES-Proc. 88881.120045/2016-01, Brazil. M.T.O. Pimenta was supported by CNPq-Proc. 442520/2014-0 and FAPESP-Proc. 2017/01756-2, Brazil.}

\begin{document}

\maketitle
\begin{abstract}
Results about existence of a signed ground state solution and multiple solutions (if $f$ is odd with respect to the second variable) are proven for a class of asymptotically linear elliptic problems involving a Carath\'eodory type nonlinearity satisfying assumptions weaker than (fF) in \cite{GLZ} or $(F_{2})_{+}$ in \cite{CM}. A close relation between the behaviour of $\lim_{|t|\to \infty}f(x, t)/|t|$ and the number of solutions is stablished.
\end{abstract}
\maketitle

\section{Introduction}


We are interested in studying the existence of ground state and other nontrivial solutions to the following semilinear problem
\begin{equation}\label{P}\tag{P}
\left \{ \begin{array}{ll}
-\Delta u = f(x, u) & \mbox{in $\Omega$,}\\
u=0 & \mbox{on $\partial\Omega$,}
\end{array}\right.
\end{equation}
where $\Omega \subset\R^{N}$ is a bounded smooth domain, $N \geq 1$ and $f:\Omega\times \R\to\R$ is a Carath\'eodory function which is assumed to satisfy the following assumptions: 
\begin{itemize}
\item[$(f_1)$] $t\mapsto f(x, t)/|t|$ is increasing, a.e. in $\Omega$, $\alpha(x):=\lim_{t\to 0}f(x, t)/|t|\geq 0$ and $\eta(x):=\lim_{|t|\to \infty}f(x, t)/|t|$ are $L^{\infty}(\Omega)-$functions;
\item[$(f_{2})$] $\lambda_{m}(\eta)<1<\lambda_{1}(\alpha)$, for some $m\geq 1$, where $\lambda_{m}(\theta)$ denotes the $m$-th eigenvalue of the problem 
\begin{equation}\label{EP}
\left \{ \begin{array}{ll}
-\Delta u = \lambda \theta(x)u & \mbox{in $\Omega$,}\\
u=0 & \mbox{on $\partial\Omega$,}
\end{array}\right.
\end{equation}
 \end{itemize}
which exists due to $(f_{1})$ (see for instance \cite{DeFig}). Moreover, if $\alpha=0$ we define $\lambda_{1}(\alpha)=\infty$. We denote $\lambda_{j}$ an eigenvalue of \eqref{EP} with $\theta=1$.


Problems satisfying conditions as $(f_{1})$ are known in the literature as asymptotically linear and can be classified as resonant at infinity (if $\eta(x)=\lambda_{j}$, for some $j$) or non-resonant at infinity (if $\eta(x)\neq \lambda_{j}$, for all $j\in\N$). In fact the resonant case is subdivided depending on how small at infinity is the function $g(x,t) = \eta(x) t - f(x,t)$. The {\it worst} situation is when, a.e. in $\Omega$, $\lim_{|t| \to\infty}g(x,t)= 0$ and $\lim_{|t| \to\infty} \int_0^t g(x,s)ds = b(x)<\infty$ and is called the strong resonant case. One of the very first works dealing with this situation, in the case $b(x)=b\in\R$, is \cite{BBF} where Bartolo, Benci and Fortunato show the existence and multiplicity of solutions to strong resonant problems in the presence of some symmetry in the autonomous nonlinearity. Their proofs are based on a deformation theorem and pseudo-index theory.

Besides \cite{BBF}, there exist so many other works dealing with problem (\ref{P}). For instance, in the non-resonant case with $m=1$, Amann and Zehnder, in \cite{AZ}, proved that problem \eqref{P} has at least one nontrivial solution. Considering yet the non-resonant case with $m\geq 1$, Ahmad, in \cite{Ah}, proved the existence of at least two nontrivial solutions. Other interesting results on the same issue can be found in \cite{Am},  \cite{Da}, \cite{LW} and \cite{Str1}.  As far as the resonant case is regarded, the existence of nontrivial solution was studied in \cite{DfM} by de Figueiredo and Miyagaki and in \cite{LW1} by Li and Willem. Multiplicity results for problem \eqref{P} in the resonant case were also investigated in \cite{LW} by Li and Willem, \cite{LZ} by Liu and Zou, \cite{Su} by Su and in \cite{SZ} by Su and Zhao. It is essential to point out that in all the references above, the nonlinearity $f$ is assumed to be differentiable (or even $C^{1}$) in the second variable, being this assumption crucial in their arguments.

Recently, Li and Zhou in \cite{GLZ} considered the particular case of $(f_{1})$ in which $\alpha(x)=0$ and $\eta(x)=\eta$ is a real number. Among other things, the authors proved that in the resonant case $\eta=\lambda_{m+1}$, for some $m\in\N$, problem \eqref{P} has at least $1+\sum_{j=2}^{m}d_{j}$ pairs of nontrivial solutions, provided that $f(x, .)$ is odd with respect the second variable and 
\begin{equation}\label{fF}\tag{fF}
\beta(x):=\lim_{|t|\to\infty}\left[(1/2)f(x, t)t-F(x, t)\right]=+\infty \ \mbox{uniformly in $x\in\Omega$},
\end{equation}
where $d_{j}$ denotes the dimension of $j$-th eigenspace associated to $\lambda_{j}$ and $F$ is the primitive of $f$. Condition \eqref{fF} is just the weakest assumption  of a series of conditions which improve the well known Ambrosetti-Rabinowitz condition and it has appeared in other papers, for example, see condition $(F_{2})_{+}$ in \cite{CM}.

In the present paper some progresses are obtained regarding the previous works. In what follows, we enumerate the main contributions: (1) In order to get our main results in the resonant case, we are not requiring condition \eqref{fF}. Instead, we use the conditions \eqref{ttttt} or \eqref{finito}, which are weaker than \eqref{fF}; (2) Since $\beta$ depend on $x$, our assumptions are more general that most of those ones in the previous papers which usually require that $\beta$ be constant or infinity. Such a restriction in some previous works holds mainly due to hypothesis \eqref{fF}; (3) We provide an unified approach to deal at once with the non-resonant case, the strong and the non-strong resonant cases.

Our main results reads as follows:



\begin{theorem}\label{teo1}
Suppose that $f$ satisfies $(f_{1})-(f_{2})$ and 
\begin{equation}\label{ttttt}\tag{$\beta$}
\essinf\limits_{x\in\Omega}\beta(x)> \frac{|\eta|_{\infty} \tau^{2}}{2\lambda_{1}(\eta-\alpha)S(\Omega)^{N/2}},
\end{equation}
where $\tau$ is defined in Proposition \ref{proposition2}$(A_{1})$. Then, there exists a signed ground-state solution (in the mountain pass level) for problem $(\ref{P})$.
\end{theorem}

\begin{theorem}\label{teore2}
Suppose that $f(x,\cdot)$ is odd a.e. in $\Omega$, satisfies $(f_{1})-(f_{2})$ and 
\begin{equation}\label{finito}\tag{$\beta_{m}$}
\essinf\limits_{x\in\Omega}\beta(x)> \frac{|\eta|_{\infty} \tau_{m}^{2}}{2\lambda_{1}(\eta-\alpha)S(\Omega)^{N/2}},
\end{equation} 
where $\tau_{m}$ is defined in \eqref{numer}. Then, problem \eqref{P} has at least $s_{m}$ pairs of nontrivial solutions, where $s_{m}=1+\sum_{j=2}^{m}d_{j}$ is the sum of the dimensions $d_{j}$ of the first $m$ eigenspaces $V_{j}$ associated to $(-\Delta,H^1_0(\Omega))$.
\end{theorem}

Our approach is based on the Nehari method which consists in minimizing the energy functional $I$ over the so called Nehari manifold $\mathcal{N}$, a set which contains all the nontrivial solutions of the problem. Although this method has been carefully treated by Szulkin and Weth in \cite{SW} for the case of nonlinearities which satisfies superquadraticity conditions at infinity, it is not a trivial task to use it to treat asymptotically linear problems. Just to cite the main difficulties to apply the method of the Nehari manifold to problems involving asymptotically linear nonlinearities, we are going to describe it here in a very superficial way. Roughly speaking, the method consists in proving the existence of a homeomorphism $m$ between $\mathcal{N}$ and a submanifold $\mathcal{M}$ of $H_{0}^{1}$. Despite the absence of a differentiable structure in $\mathcal{N}$, such a homeomorphism allows us to define a $C^1$-functional $\Psi$ on $\mathcal{M}$ with very useful properties. For example, if $u$ is a critical point of $\Psi$ then $m(u)$ is a critical point of $I$ and if $\{u_{n}\}$ is a $(PS)_c$ sequence for $\Psi$ then $\{m(u_{n})\}$ is $(PS)_c$ sequence for $I$.

Differently of the case of nonlinearities satisfying superquadraticity conditions at infinity, in which $\mathcal{M}$ is the unit sphere $\mathcal{S}$ of $H_{0}^{1}$, it is not clear exactly how is $\mathcal{M}$ in our case. After a careful study (see Lemma \ref{lemma1}, Propositions \ref{proposição1} and \ref{proposition2}) we have proven that $\mathcal{M}=\mathcal{S}_{\mathcal{A}}:=\mathcal{S}\cap\mathcal{A}$ is a noncomplete submanifold of $H_{0}^{1}$, where $\mathcal{A}:=\left\{u\in H_{0}^{1}(\Omega): \|u\|^{2}<\int_{\Omega}\eta(x)u^{2}dx\right\}$. This fact brings additional problems. Indeed, it is important to assure that minimizing sequences $\{u_{n}\}$ for $\Psi$ are not near the boundary of $\mathcal{S}_{\mathcal{A}}$. In \cite{SW}, this step is strongly based in the fact that $f$ has a superquadratic growth at infinity, what implies that $\{\Psi(u_{n})\}$ tends to infinity as the distance from $\{u_{n}\}$ to the boundary tends to zero. In our case, the behaviour of $\{\Psi(u_{n})\}$ at infinity, as $dist(u_{n}, \partial \mathcal{S}_{\mathcal{A}})\to 0$, is indefinite. This fact makes difficult, for example, to know how to extend $\Psi$ to $\overline{\mathcal{S}_{\mathcal{A}}}$ in order to apply the Ekeland variational principle, which is crucial to prove that $\{u_{n}\}$ can be seen as a Palais-Smale sequence. In order to answer these fundamental questions, we have provided a very useful technical result, see Proposition \ref{Fatou}, which we believe is important by itself. Finally, it is also a delicate task to prove that $\Psi$ satisfies the $(PS)_c$ condition under conditions \eqref{ttttt} and \eqref{finito} (see Proposition \ref{main2}).


%
%
%

The paper is organized as follows. In Section \ref{sec:prelim} we present the variational background and a useful technical proposition. In Section \ref{sec:Nehari} we study deeply the Nehari manifold and some of its topological features. In Section \ref{sec:ground} we state and prove our main result of existence of ground-state solution. In Section \ref{sec:multiplicity} we introduce some classical tools of genus theory and we prove our main result of multiplicity of solutions.

\medskip

\section{Preliminaries}\label{sec:prelim}

Our main goal in this section is to introduce the variational background for \eqref{P} and to prove a very useful proposition which provides us a slight generalization of the  Fatou Lemma. 

We stand for $I:H_{0}^{1}(\Omega)\to\R$ the energy functional associated to problem \eqref{P}, given by
$$
I(u)=\frac{1}{2}\|u\|^{2}-\int_{\Omega}F(x, u)dx,
$$
where $\displaystyle \|u\|^{2} = \int_\Omega |\nabla u|^2 dx$ and $\displaystyle  F(x, t)=\int_{0}^{t}f(x, s)ds$. It is well known that $I\in C^{1}(H_{0}^{1}(\Omega), \R)$ and
$$
I'(u)\varphi=\int_{\Omega}\nabla u\nabla \varphi dx-\int_{\Omega}f(x, u)\varphi dx.
$$
In this way, critical points of $I$ are weak solutions of \eqref{P}.

The Nehari manifold associated to the functional $I$ is the set
$$
\mathcal{N}=\{u\in H_{0}^{1}(\Omega)\backslash\{0\}:\|u\|^{2}=\int_{\Omega}f(x, u)udx\}.
$$
Since $f$ is just a Carath\'eodory function, we cannot ensure that $\mathcal{N}$ is a smooth manifold. In this paper, $\mathcal{S}$ denotes the unit sphere in $H_{0}^{1}(\Omega)$
and
$$
\mathcal{A}:=\left\{u\in H_{0}^{1}(\Omega): \|u\|^{2}<\int_{\Omega}\eta(x)u^{2}dx\right\}.
$$ 

\medskip

Throughout this paper we denote by $e_{j}$ a normalized (in $H_{0}^{1}(\Omega)$ norm) eigenfunction associated to $\lambda_{j}(\eta)$ and we use the symbol $[u\neq 0]$ to denote the set $\{x\in\Omega: u(x)\neq 0\}$. Moreover, $|A|$ will always denote the Lebesgue measure of a measurable set $A\subset\R^{N}$.

\begin{lemma}\label{lemma1}
If $f$ satisfies $(f_{1})-(f_{2})$, the following claims hold:
\begin{enumerate}
\item[$(i)$] The set $\mathcal{A}$ is open and nonempty;
\item[$(ii)$] $\partial \mathcal{A}=\{u\in H_{0}^{1}(\Omega): \|u\|^{2}=\int_{\Omega}\eta(x)u^{2}dx\}$;
\item[$(iii)$] $\mathcal{A}^{c}=\{u\in H_{0}^{1}(\Omega): \|u\|^{2}\geq\int_{\Omega}\eta(x)u^{2}dx\}$;
\item[$(iv)$] $\mathcal{N}\subset \mathcal{A}$;
\item[$(v)$] $\mathcal{S}\cap\mathcal{A}\neq\emptyset$.
\end{enumerate}
\end{lemma}
\begin{proof} ($i$) By $(f_{2})$, if $u$ is an eigenfunction associated to $\lambda_{j}(\eta)$, for some $j\in\{1, \ldots, m\}$, then $u$ belongs to $\mathcal{A}$. Moreover, $\mathcal{A}=\varphi^{-1}(-\infty, 0)$ where $\varphi:H_{0}^{1}(\Omega)\to \mathbb{R}$ is the continuous function $\varphi(u)=\|u\|^{2}-\int_{\Omega}\eta(x)u^{2}dx$. The items ($ii$) and ($iii$) are immediate.

 ($iv$) If $u\in\mathcal{N}$ then, 
$$
\|u\|^{2}=\int_{[u\neq 0]}\left[\frac{f(x, u)}{u}\right]u^{2}dx.
$$
By $(f_{1})$, we conclude that 
$$
\|u\|^{2}<\int_{\Omega}\eta(x)u^{2}dx.
$$

($v$) It is sufficient to choose a normalized eigenfunction $e_{j}$ associated to $\lambda_{j}(\eta)$ for any $j\in\{1, \ldots, m\}$. Clearly, one has that $e_{j}\in\mathcal{S}\cap\mathcal{A}$.
\end{proof}

\medskip

In the sequel we will denote $\mathcal{S}_{\mathcal{A}}:=\mathcal{S}\cap\mathcal{A}$. Observe that being $\mathcal{S}$ a $C^{1}$-submanifold of $H^{1}_{0}(\Omega)$, by $(i)$ and $(v)$ of Lemma $\ref{lemma1}$, $\mathcal{S}_{\mathcal{A}}$ is an open set in $\mathcal{S}$ and, therefore, it is also a $C^{1}$-submanifold of $H^{1}_{0}(\Omega)$, which is noncomplete if $1\leq \lambda_{m+k}(\eta)$ for some $k\geq 1$. In fact, in this case we have $e_{j}\in \mathcal{S}_\mathcal{A}^{c}$ for all $j\geq m+k$. 

From now on, $\partial \mathcal{S}_\mathcal{A}$ denotes the relative boundary of $\mathcal{S}_\mathcal{A}$ as a topological subspace of $\mathcal{S}$ and $\mathcal{S}_\mathcal{A}^c$ is the set $\mathcal{S}\backslash \mathcal{S}_\mathcal{A}$. Moreover, from $(ii)$ and $(iii)$, it is clear that $\partial \mathcal{S}_{\mathcal{A}}=\{u\in \mathcal{S}: 1=\int_{\Omega}\eta(x)u^{2}dx\}$ and $\mathcal{S}_{\mathcal{A}}^{c}=\{u\in \mathcal{S}: 1\geq\int_{\Omega}\eta(x)u^{2}dx\}$.

\medskip

\begin{lemma}\label{limita}
Following inequality holds:
$$
\inf_{u\in \partial\mathcal{S}_{\mathcal{A}}}|[u\neq 0]|\geq \left( S(\Omega)/|\eta|_{\infty}\right)^{N/2},
$$
where $1/S(\Omega)$ is the best constant of the continuous embedding from $H_{0}^{1}(\Omega)$ into $L^{2^{\ast}}(\Omega)$.
\end{lemma}

\begin{proof}
By using H\"older inequality, it follows that, for each $u\in \partial\mathcal{S}_{\mathcal{A}}$
$$
1\leq |\eta|_{\infty}\int_{[u\neq 0]}u^{2}dx\leq |\eta|_{\infty}|u|_{2^{\ast}}^{2}|[u\neq 0]|^{2/N}.
$$
By continuous Sobolev embedding from $H_{0}^{1}(\Omega)$ into $L^{2^{\ast}}(\Omega)$, 
$$
1\leq |\eta|_{\infty}(1/S(\Omega))|[u\neq 0]|^{2/N}.
$$
Therefore,
$$
|[u\neq 0]|\geq (S(\Omega)/|\eta|_{\infty})^{N/2}, \ \forall \ u\in \partial\mathcal{S}_{\mathcal{A}}.
$$
The result is proven.
\end{proof}

Next lemma provides some consequences of hypothesis $(f_{1})$ which will be used later on.

\begin{lemma}\label{lemma2}
Suppose $(f_{1})$ holds. Then, a.e. in $\Omega$,
\begin{enumerate}
\item[$(i)$] $t\mapsto (1/2)f(x, t)t-F(x, t)$ is increasing in $(0,\infty)$ and decreasing in $(-\infty, 0)$;
\item[$(ii)$] $t\mapsto F(x, t)/t^{2}$ is increasing in $(0,\infty)$ and decreasing in $(-\infty, 0)$;
\item[$(iii)$] $f(x, t)/t>2F(x, t)/t^{2}$ for all $t\in \R\backslash\{0\}$.
\end{enumerate}
\end{lemma}

\begin{proof}
$(i)$ Let $t_{1}>t_{2}>0$. Then, a.e. in $\Omega$,
\begin{eqnarray*}
\frac{1}{2}f(x, t_{1})t_{1}-F(x, t_{1}) &=&\frac{1}{2}f(x, t_{1})t_{1}-F(x, t_{2})-\int_{t_{2}}^{t_{1}}\left[\frac{f(x, s)}{s}\right]sds\\
&>&\frac{1}{2}f(x, t_{1})t_{1}-F(x, t_{2})-\frac{f(x, t_{1})}{t_{1}}\int_{t_{2}}^{t_{1}}sds\\
&=& \frac{1}{2}f(x, t_{1})t_{1}-F(x, t_{2})-\frac{f(x, t_{1})}{t_{1}}\frac{(t_{1}^{2}-t_{2}^{2})}{2}\\
&=&\frac{f(x, t_{1})}{t_{1}}\frac{t_{2}^{2}}{2}-F(x, t_{2})\\
&> &\frac{1}{2}f(x, t_{2})t_{2}-F(x, t_{2}),
\end{eqnarray*}
where it was used $(f_{1})$ in the last two inequalities. The other case is analogous. Items $(ii)$ and $(iii)$ follows from $(i)$.
\end{proof}

\begin{remark}\label{rem0}
It follows from Lemma \ref{lemma2}$(i)$ that $\beta:\Omega\to(0,\infty]$ is well defined by
$$
\beta(x):=\lim_{|t|\to\infty}\left[(1/2)f(x, t)t-F(x, t)\right].
$$

\end{remark}

To finish this section, we will provide a technical result. In fact, item $(i)$ allows us to conclude a kind of lower semicontinuity for the map $u\mapsto \chi_{[u\neq 0]}$. Moreover, item $(iii)$ gives us an interesting generalization of the well known Fatou Lemma (to recover the Fatou Lemma simply choose $v_n=u_{n}$) which will play a crucial role in the proof of Theorem \ref{teo1}.

\begin{proposition}\label{Fatou}
Let $\{u_{n}\}$ be a sequence of measurable functions $u_{n}:\Omega\to \R$. Then,
\begin{enumerate}
\item[$(i)$] 
$$
\chi_{\left[\liminf\limits_{n\to\infty} u_{n}\neq 0\right]}(x)\leq \liminf\limits_{n\to\infty}\chi_{[u_{n}\neq 0]}(x) \ \mbox{in $\Omega$};
$$

\item[$(ii)$] If, in addition, $u_{n}(x)\to u(x)$ a. e. in $\Omega$, then  
$$
\lim_{n\to\infty}\chi_{[u_{n}\neq 0]}(x)=\chi_{\left[u\neq 0\right]}(x) \ \mbox{a. e. in $[u\neq 0]$};
$$

\item[$(iii)$] Let $\{v_{n}\}$ be a sequence of nonnegative measurable functions $v_{n}:\Omega\to \R$. Then
$$
\int_{\left[\liminf\limits_{n\to\infty} u_{n}\neq 0\right]}\left[\liminf\limits_{n\to\infty} v_{n}(x)\right] dx\leq \liminf\limits_{n\to\infty} \int_{[u_{n}\neq 0]}v_{n}(x) dx.
$$
In particular,
\begin{equation}\label{conclusion}
\left|\left[\liminf\limits_{n\to\infty} u_{n}\neq 0\right]\right|\leq \liminf\limits_{n\to\infty} |[u_{n}\neq 0]|.
\end{equation}
\end{enumerate}
\end{proposition} 

\begin{proof}
$(i)$ 
Since $[w\neq 0]=[|w|\neq 0]$ for all measurable function $w$, it is sufficient to prove that 
$$
\chi_{\left[|\liminf\limits_{n\to\infty} u_{n}|\neq 0\right]}(x)\leq \liminf\limits_{n\to\infty}\chi_{[|u_{n}|\neq 0]}(x).
$$
For this, let us define $u:=\liminf\limits_{n\to\infty} u_{n}$ and $g:\Omega\to\{0, 1\}$ by 
$$
g(x)=\liminf\limits_{n\to\infty}\chi_{[|u_{n}|\neq 0]}(x).
$$
If $g\equiv 1$, there is nothing to be proven. Otherwise, it is sufficient to prove that if $g(x)=0$, then $\chi_{[|u|\neq 0]}(x)=0$. Indeed, observe that if $g(x)=0$ then there exists a subsequence $u_{n_{k}}$ where $\{n_{k}\}\subset\N$ depends on $x$, such that
$$
\chi_{[|u_{n_{k}}|\neq 0]}(x)=0, \ \forall \ k\in\N.
$$
Equivalently,
$$
|u_{n_{k}}(x)|=0, \ \forall \ k\in\N.
$$
Passing to the lower limit as $k$ goes to infinity, we obtain
$$
0\leq |u(x)|= \liminf\limits_{n\to\infty} |u_{n}(x)|\leq \liminf\limits_{k\to\infty} |u_{n_k}(x)|=0,
$$
or yet
$$
\chi_{[|u|\neq 0]}(x)=0.
$$

$(ii)$ Taking into account the previous item, it is sufficient to prove that
$$
\limsup\limits_{n\to\infty}\chi_{[u_{n}\neq 0]}(x)\leq \chi_{\left[u\neq 0\right]}(x) \ \mbox{a. e. in $[u\neq 0]$}. 
$$
Indeed, there exists a set $\hat{\Omega}\subset\Omega$ with measure zero, such that
$$
u_{n}(x)\to u(x), \ \forall \ x\in \Omega\backslash\hat{\Omega}.
$$
Thus, for each $x\in (\Omega\backslash\hat{\Omega})\cap [u\neq 0]$, there exists $n(x)$ such that if $n\geq n(x)$ then $u_{n}(x)\neq 0$, or equivalently, $\chi_{[u_{n}\neq 0]}(x)=1$ for all $n\geq n(x)$. Therefore,
$$
\limsup\limits_{n\to\infty}\chi_{[u_{n}\neq 0]}(x)=1=\chi_{\left[u\neq 0\right]}(x), \ \forall \ x\in (\Omega\backslash\hat{\Omega})\cap [u\neq 0].
$$

$(iii)$ Denote $v:=\liminf\limits_{n\to\infty} v_{n}$. It follows from item $(ii)$ and properties of lower limit that
$$
v(x)\chi_{[u\neq 0]}(x)\leq \liminf\limits_{n\to\infty} \left[v_{n}(x)\chi_{[u_{n}\neq 0]}(x)\right] \ \mbox{a.e. in $\Omega$}.
$$
By integrating last inequality and using the classical Fatou Lemma, the result follows. To conclude \eqref{conclusion}, simply choose $v_{n}=v=1$. 

\end{proof}

\medskip


\section{Topological aspects of the Nehari manifold}\label{sec:Nehari}

The main goal of this section is to study some topological features of the Nehari manifold and the behavior of the energy functional $I$ on $\mathcal{N}$.

\begin{proposition}\label{proposição1}
Suppose that $f$ verifies $(f_1)-(f_{2})$ and let $h_{u}:[0,\infty)\rightarrow \R$ be defined by
$h_{u}(t)=I(tu)$.
\begin{enumerate}
\item[$(i)$] For each $u\in \mathcal{A}$, there exists a unique $t_{u}>0$ such that $h_{u}'(t)>0$ in $(0,t_{u})$, $h_{u}'(t_{u})=0$ and $h_{u}'(t)<0$ in $(t_{u}, \infty)$. Moreover, $tu\in \mathcal{N}$ if, and only if, $t=t_{u}$;

\item[$(ii)$] for each $u\in \mathcal{A}^{c}$, $h_{u}'(t)>0$ for all $t\in (0, \infty)$.
\end{enumerate}
\end{proposition}

\begin{proof} 
($i$) First observe that $h_{u}(0)=0$. Moreover, for each $u\in \mathcal{A}$, we have
\begin{equation}
\frac{h_{u}(t)}{t^{2}}= \frac{1}{2}\|u\|^{2}-\int_{[u\neq 0]}\left[\frac{F(x, tu)}{(tu)^{2}}\right]u^{2}dx.
\end{equation}
Thus, from $(f_{1})-(f_{2})$, L'Hospital rule and Lebesgue dominated convergence theorem, it follows that
$$
\lim_{t\to 0}\frac{h_{u}(t)}{t^{2}}=\frac{1}{2}\left(\|u\|^{2}-\int_{\Omega}\alpha(x)u^{2}dx\right)>0
$$
and
$$
\lim_{t\to \infty}\frac{h_{u}(t)}{t^{2}}=\frac{1}{2}\left(\|u\|^{2}-\int_{\Omega}\eta(x)u^{2}dx\right)<0.
$$
Showing that 
$$
h_{u}(t)=\frac{h_{u}(t)}{t^{2}}t^{2}
$$
is positive for $t$ small and

$$
\lim_{t\to \infty}h_{u}(t)=\lim_{t\to \infty}\frac{h_{u}(t)}{t^{2}}t^{2}=-\infty.
$$
Since $h_{u}$ is a continuous function, previous arguments implies that there exists a global maximum point $t_{u}>0$ of $h_{u}$. Now, we are going to show that $t_{u}$ is the unique critical point of $h_{u}$. In fact, supposing that there exist $t_{1}>t_{2}>0$ such that $h_{u}'(t_{1})=h_{u}'(t_{2})=0$, we obtain
$$
0=\int_{[u\neq 0]}\left[\frac{f(x, t_{1}u)}{t_{1}u}-\frac{f(x, t_{2}u)}{t_{2}u}\right]u^{2}dx,
$$
and, by $(f_{1})$, $t_{1}=t_{2}$. The result follows.

($ii$) If $u\in \mathcal{A}^{c}$, then $\|u\|^{2}\geq \int_{\Omega}\eta(x)u^{2}dx$. Thus, it follows from $(f_{1})$ that
$$
\frac{h_{u}'(t)}{t}=\|u\|^{2}-\int_{[u\neq 0]}\frac{f(x, tu)}{tu}u^{2}dx\geq \int_{[u\neq 0]}\left[\eta(x)-\frac{f(x, tu)}{tu}\right]u^{2}dx>0, \ \forall \ t>0.
$$ 
Consequently, $h_{u}'(t)=t(h_{u}'(t)/t)>0$ for all $t\in (0,\infty)$.
\end{proof}

\medskip

\begin{remark}\label{rem1}
It is an immediate consequence of previous proposition that, for each $u\in \mathcal{A}$ and $s\in (0, \infty)$, $t_{su}=t_{u}/s$. Moreover, it is clear that $u\in \mathcal{N}$ if, and only if, $t_{u}=1$.
\end{remark}

\medskip

\begin{proposition}\label{proposition2}
If $f$ verifies $(f_{1})-(f_{2})$, the following claims hold:
\begin{enumerate}
\item[$(A_{1})$] $\tau:=\inf_{u\in \mathcal{S}_{\mathcal{A}}}t_{u}>0$;

\item[$(A_{2})$] $\zeta_\mathcal{W}:=\max_{u\in \mathcal{W}}t_{u}<\infty$, for all compact set $\mathcal{W}\subset \mathcal{S}_{\mathcal{A}}$;

\item[$(A_{3})$] Map
$\widehat{m}:\mathcal{A}\rightarrow \mathcal{N}$ given by
$\widehat{m}(u)=t_{u}u$ is continuous and
$m:=\widehat{m}_{|_{\mathcal{S}_{\mathcal{A}}}}$ is a homeomorphism between
$\mathcal{S}_{\mathcal{A}}$ and $\mathcal{N}$. Moreover, $m^{-1}(u)=u/\|u\|$.
\end{enumerate}
\end{proposition}

\begin{proof}
($A_{1}$) Suppose there exists $\{u_{n}\}\subset \mathcal{S}_{\mathcal{A}}$ such that $t_{n}:=t_{u_{n}}\to 0$. In this case, we get $u\in H_{0}^{1}(\Omega)$ such that $u_{n}\rightharpoonup u$ in $H_{0}^{1}(\Omega)$. Observe that
\begin{equation}\label{padrao}
1=\int_{[u\neq 0]}\left[\frac{f(x, t_{n}u_{n})}{t_{n}u_{n}}\right]\chi_{[u_{n}\neq 0]}u_{n}^{2} dx + \int_{[u=0]}\left[\frac{f(x, t_{n}u_{n})}{t_{n}u_{n}}\right]\chi_{[u_{n}\neq 0]}u_{n}^{2} dx, \ \forall \ n\in\N.
\end{equation}
By $(f_1)$ and Lebesgue dominated convergence theorem,
\begin{equation}\label{zero}
\int_{[u=0]}\left[\frac{f(x, t_{n}u_{n})}{t_{n}u_{n}}\right]\chi_{[u_{n}\neq 0]}u_{n}^{2} dx\to 0.
\end{equation}
On the other hand, by $(f_1)$, Proposition \ref{Fatou}$(ii)$ and Lebesgue dominated convergence theorem
\begin{equation}\label{nonzero}
\int_{[u\neq 0]}\left[\frac{f(x, t_{n}u_{n})}{t_{n}u_{n}}\right]\chi_{[u_{n}\neq 0]}u_{n}^{2} dx\to \int_{\Omega}\textcolor{blue}{\alpha(x)}u^{2}dx.
\end{equation}
Thus, by \eqref{zero} and \eqref{nonzero}, passing to the limit as $n$ goes to infinity in \eqref{padrao}, we get
$$
1=\int_{\Omega}\alpha(x)u^{2}dx.
$$
If $\alpha=0$, we have a clear contradiction. Otherwise, the inequality
$$
1\leq (1/\lambda_{1}(\alpha))\|u\|^{2}\leq 1/\lambda_{1}(\alpha),
$$
contradicts $(f_{2})$.


($A_{2}$) Suppose there exists $\{u_{n}\}\subset \mathcal{W}$ such that $t_{n}:=t_{u_{n}}\to \infty$. Since $\mathcal{W}$ is compact, passing to a subsequence, we obtain $u\in \mathcal{W}$ such that $u_{n}\to u$ in $H_{0}^{1}(\Omega)$. Hence, passing to the lower limit as $n\to\infty$ in
$$
1=\|u_{n}\|^{2}=\int_{[u_{n}\neq 0]}\frac{f(x, t_{n}u_{n})}{t_{n}u_{n}}u_{n}^{2}dx, \ \forall \ n\in\N,
$$
it follows from Proposition \ref{Fatou}$(iii)$ that
$$
1=\|u\|^{2}\geq \int_{\Omega}\eta(x)u^{2}dx.
$$
The last inequality implies that $u\in \mathcal{S}_{\mathcal{A}}^{c}$, leading us to a contradiction since $u\in \mathcal{W}\subset\mathcal{S}_{\mathcal{A}}$.

($A_{3}$) We first show that $\widehat{m}$ is continuous. Let $\{u_{n}\}\subset \mathcal{A}$ and $u\in\mathcal{A}$, be such that $u_{n}\to u$ in $H_{0}^{1}(\Omega)$. From Remark \ref{rem1} ($\widehat{m}(tw)=\widehat{m}(w)$ for all $w\in \mathcal{A}$ and $t>0$), we can consider $\{u_{n}\}\subset S_{\mathcal{A}}$. Thus,
\begin{equation}\label{conv1}
t_{n}=t_{n}\|u_{n}\|^{2}=\int_{\Omega}f(x, t_{n}u_{n})u_{n}dx,
\end{equation}
where $t_{n}:=t_{u_{n}}$. From ($A_{1}$) and $(A_{2})$, it follows that, up to a subsequence, $t_{n}\to t>0$. Thence, passing to the limit as $n\to\infty$ in \eqref{conv1}, we have
$$
t=t\|u\|^{2}=\int_{\Omega}f(x, tu)u dx,
$$
showing that $\widehat{m}(u_{n})=t_{n}u_{n}\to tu=\widehat{m}(u)$. The second part  of ($A_{3}$) is immediate.
\end{proof}

\begin{lemma}\label{lemma3}
The functional $I$ is bounded from below on $\mathcal{N}$, more specifically,
$$
I(u)>0,
$$
for all $u\in\mathcal{N}$.
\end{lemma}

\begin{proof}
For any $u\in \mathcal{S}_{\mathcal{A}}$, we get
$$
I(t_{u}u)=\int_{\Omega}\left[\frac{1}{2}\frac{f(x, t_{u}u)}{t_{u}u}-\frac{F(x, t_{u}u)}{(t_{u}u)^{2}}\right](t_u u)^{2}dx.
$$
The result follows from Lemma \ref{lemma2}$(iii)$.
\end{proof}

Now we set the maps $\widehat{\Psi}:\mathcal{A}\rightarrow\mathbb{R} \ \mbox{and} \ \Psi:\mathcal{S}_{\mathcal{A}}\rightarrow \mathbb{R}$, by 
$$
\widehat{\Psi}(u)=I(\widehat{m}(u)) \ \mbox{and} \ \Psi:=\widehat{\Psi}_{|_{\mathcal{S_{\mathcal{A}}}}}.
$$ 

In the next result we present some properties of $\widehat{\Psi}$ and $\Psi$. The proof of such a result is a consequence of Proposition \ref{proposition2} and the details can be found in \cite{SW}. 



\begin{proposition}\label{proposition3}
Suppose that $f$ verifies $(f_{1})-(f_{2})$. Then,
\begin{enumerate}
\item[$(i)$] $\widehat{\Psi}\in C^{1}(\mathcal{A}, \mathbb{R})$ and
$$
\widehat{\Psi}'(u)v=\frac{\|\widehat{m}(u)\|}{\|u\|}I'(\widehat{m}(u))v,
\ \forall u\in \mathcal{A} \ \mbox{and} \ \forall v\in H_{0}^{1}(\Omega).
$$

\item[$(ii)$] $\Psi\in C^{1}(\mathcal{S}_{\mathcal{A}}, \mathbb{R})$ and
$$
\Psi'(u)v=\|m(u)\|I'(m(u))v, \ \forall v\in T_{u}\mathcal{S}_{\mathcal{A}}.
$$

\item[$(iii)$] If $\{u_{n}\}$ is a $(PS)_{c}$ sequence for $\Psi$ then  $\{m(u_{n})\}$ is a $(PS)_{c}$ sequence for $I$.
If $\{u_{n}\}\subset \mathcal{N}$ is a bounded $(PS)_{c}$ sequence for $I$ then $\{m^{-1}(u_{n})\}$ is a $(PS)_{c}$ sequence for $\Psi$.

\item[$(iv)$] $u$ is a critical point of $\Psi$ if, and only if, $m(u)$ is a nontrivial critical point of $I$. Moreover,
$$
c_{\mathcal{N}}:=\inf_{u\in \mathcal{N}}I(u)=\inf_{u\in \mathcal{A}}\max_{t>0}I(tu)=\inf_{u\in\mathcal{S}_{\mathcal{A}}}\max_{t>0}I(tu)=\inf_{u\in\mathcal{S}_{\mathcal{A}}}\Psi(u).
$$
\end{enumerate}
\end{proposition}

\begin{remark}\label{rem2}
It is a consequence of Lemma \ref{lemma3} that $c_{\mathcal{N}}\geq 0$. Moreover, if $c_{\mathcal{N}}$ is achieved then it is positive.
\end{remark}

Since $\mathcal{S}_{\mathcal{A}}$ can be non-complete, we need to be very careful about the behaviour of minimizing sequences for $\Psi$ near the boundary. Next result helps us in this issue.

\begin{proposition}\label{main1}
Suppose $(f_{1})-(f_{2})$ hold. If $\{u_{n}\}\subset \mathcal{S}_{\mathcal{A}}$ is such that $dist(u_{n}, \partial \mathcal{S}_{\mathcal{A}})\to 0$, then there exists $u\in H_{0}^{1}(\Omega)\backslash\{0\}$ such that $u_{n}\rightharpoonup u$ in $H_{0}^{1}(\Omega)$, $t_{u_{n}}\to\infty$ and 
$$
\liminf\limits_{n\to\infty}\Psi(u_{n})\geq\int_{[u\neq 0]}\beta(x) dx,
$$ 
where $\beta(x)$ is defined in Remark \ref{rem0}. 
\end{proposition}

\begin{proof}

Since $\{u_{n}\}\subset \mathcal{S}_{\mathcal{A}}$ is bounded, up to a subsequence, there exists $u\in H_{0}^{1}(\Omega)$ with $u_{n}\rightharpoonup u$ in $H_{0}^{1}(\Omega)$. Since $dist(u_{n}, \partial \mathcal{S}_{\mathcal{A}})\to 0$, there exists $\{z_{n}\}\subset \mathcal{S}_{\mathcal{A}}$ such that $\|u_{n}-z_{n}\|\to 0$ as $n\to\infty$. Thus,
\begin{eqnarray*}
\left| \int_{\Omega}\eta(x)u_{n}^{2}dx-1\right|&=&\left| \int_{\Omega}\eta(x)(u_{n}^{2}-z_{n}^{2})dx\right|\\
&\leq& |\eta|_{\infty}|u_{n}+z_{n}|_{2}|u_{n}-z_{n}|_{2}\\
&\leq&  (2|\eta|_{\infty}/\lambda_{1})\|u_{n}-z_{n}\|.
\end{eqnarray*}
Therefore, 
$$
\int_{\Omega}\eta(x)u_{n}^{2}dx\to 1.
$$

By using compact embedding from $H_{0}^{1}(\Omega)$ in $L^{2}(\Omega)$, it follows that
\begin{equation}\label{equal}
1=\int_{\Omega}\eta(x)u^{2}dx.
\end{equation}
Thus $u\neq 0$. Suppose by contradiction that, for some subsequence, $\{t_{u_{n}}\}$ is bounded. In this case, passing again to a subsequence, there exists $t_{0}> 0$ (see Proposition \ref{proposition2}$(A_{1})$) such that 
\begin{equation}\label{conv2}
t_{u_{n}}\to t_{0}.
\end{equation} 
It follows from \eqref{conv2} and
$$
t_n=\int_{\Omega}f(x, t_{u_{n}}u_{n})u_{n} dx, \ \forall \ n\in\N,
$$
that
$$
1=\int_{\Omega}\frac{f(x, t_{0}u)}{t_{0}u}u^{2}dx.
$$
Combining last equality and $(f_{1})$, we have
$$
1<\int_{\Omega}\eta(x)u^{2}dx.
$$
But the previous inequality contradicts \eqref{equal}. Showing that $t_{u_{n}}\to\infty$. Since the behaviour at infinity of $\{t_{u_{n}}u_{n}\}$ is indefinite in $[u=0]$, we cannot use the standard Fatou Lemma. Instead, we will use Proposition \ref{Fatou}$(iii)$. Thence,
\begin{eqnarray*}
\liminf\limits_{n\to\infty}\Psi(u_{n})&=& \liminf\limits_{n\to\infty}\int_{\Omega}\left[ \frac{1}{2}f(x, t_{u_{n}}u_{n})t_{u_{n}}u_{n}-F(x, t_{u_{n}}u_{n})\right]dx\\
&=&\liminf\limits_{n\to\infty}\int_{[u_{n}\neq 0]}\left[ \frac{1}{2}f(x, t_{u_{n}}u_{n})t_{u_{n}}u_{n}-F(x, t_{u_{n}}u_{n})\right]dx\\
&\geq & \int_{[u\neq 0]}\beta(x) dx.
\end{eqnarray*}
\end{proof}

\begin{lemma}\label{level}
Suppose $f$ satisfies $(f_{1})-(f_{2})$ and \eqref{ttttt}. Then
$$
0\leq c_{\mathcal{N}}<\inf_{u\in \partial\mathcal{S}_{\mathcal{A}}}\int_{[u\neq 0]}\beta(x)dx.
$$ 

%
\end{lemma}
\begin{proof}

It follows from $(f_{1})$ that, for each $u\in\mathcal{S}_{\mathcal{A}}$,
\begin{eqnarray}\nonumber\label{uuuuu}
c_{\mathcal{N}}\leq\Psi(u)&=&\int_{\Omega}\left[\frac{f(x,m(u))}{2 m(u)}-\frac{F(x, m(u))}{m(u)^{2}}\right]m(u)^{2}dx\\ \nonumber
&<&(1/2)\int_{\Omega}[\eta(x)-\alpha(x)]m(u)^{2} dx\\
&\leq & [1/2\lambda_{1}(\eta-\alpha)]t_{u}^{2}.
\end{eqnarray}

On the other hand, by \eqref{limita}, for each $u\in\partial\mathcal{S}_{\mathcal{A}}$
\begin{equation}\label{vvvvv}
\int_{[u\neq 0]}\beta(x) dx\geq |[u\neq 0]|\essinf\limits_{x\in\Omega}\beta(x)\geq \left( S(\Omega)/|\eta|_{\infty}\right)^{N/2}\essinf\limits_{x\in\Omega}\beta(x).
\end{equation}
The result follows now from \eqref{ttttt}, \eqref{uuuuu} and \eqref{vvvvv}.


\end{proof}

\begin{proposition}\label{main2}
Suppose $(f_{1})-(f_{2})$ and \eqref{ttttt} to hold. Then $\Psi$ satisfies the $(PS)_{c}$ condition in $\mathcal{S}_{\mathcal{A}}$, for all $c\in [c_{\mathcal{N}}, \inf_{u\in \partial\mathcal{S}_{\mathcal{A}}}\int_{[u\neq 0]}\beta(x) dx)$.
\end{proposition}

\begin{proof} 
By Proposition \ref{proposition2}$(A_{3})$ and Proposition \ref{proposition3}$(iii)$, it is sufficient to show that $I$ satisfies the $(PS)_{c}$ condition on $\mathcal{N}$ for $c\in [c_{\mathcal{N}}, \inf_{u\in\partial\mathcal{S}_{\mathcal{A}}}\int_{[u\neq 0]}\beta(x) dx)$. 

For this, let $\{u_{n}\}\subset \mathcal{N}$ be a $(PS)_{c}$ sequence for $I$. We are going to prove that $\{u_{n}\}$ is bounded in $H_{0}^{1}(\Omega)$. In fact, suppose by contradiction that, up to a subsequence, $\|u_{n}\|\to\infty$. Define $v_{n}:=u_{n}/\|u_{n}\|=m^{-1}(u_{n})\in \mathcal{S}_{\mathcal{A}}$. Thus $\{v_{n}\}$ is bounded in $H_{0}^{1}(\Omega)$ and 
\begin{equation}\label{bound}
\Psi(v_{n})\to c.
\end{equation} 
Consequently, there exists $v\in H_{0}^{1}(\Omega)$ such that 
\begin{equation}\label{weak}
v_{n}\rightharpoonup v \ \mbox{in \ $H_{0}^{1}(\Omega)$}.
\end{equation}

Suppose $v=0$. Since $\{\Psi(v_{n})\}$ is bounded, it follows that there exists $C>0$ such that
\begin{equation}\label{6}
C>\Psi(v_{n})=I(t_{v_{n}}v_{n})\geq I(tv_{n})=(1/2)t^2-\int_{\Omega}F(x, tv_{n})dx, \ \forall \ t>0.
\end{equation}
From $(f_{1})-(f_{2})$ and compact embedding, passing to the limit of $n\to\infty$ in $(\ref{6})$, we get 
$$
C\geq (1/2)t^{2}, \ \forall \ t>0,
$$
a clear contradiction. Thereby, we conclude that $v\neq 0$. 

Since $\{u_{n}\}\subset\mathcal{N}$ is a $(PS)_{c}$ sequence for $I$, we get
$$
o_{n}(1)+\int_{\Omega}\nabla u_{n}\nabla w dx=\int_{\Omega}f(x, u_{n})w dx, \ \forall \ w\in H_{0}^{1}(\Omega).
$$
Dividing last equality by $\|u_{n}\|$, we have
\begin{eqnarray}\nonumber\label{cuenta}
&&o_{n}(1)+\int_{\Omega}\nabla v_{n}\nabla w dx=\\
&&\int_{[v\neq 0]}\left[\frac{f(x, \|u_{n}\|v_{n})}{\|u_{n}\|v_{n}}\right]\chi_{[v_{n}\neq 0]}(x)v_{n}w dx+\int_{[v= 0]}\left[\frac{f(x, \|u_{n}\|v_{n})}{\|u_{n}\|v_{n}}\right]\chi_{[v_{n}\neq 0]}(x)v_{n}w dx.
\end{eqnarray}
By $(f_{1})$, \eqref{weak} and Lebesgue dominated convergence theorem it follows that
\begin{equation}\label{ahora}
\int_{[v= 0]}\left[\frac{f(x, \|u_{n}\|v_{n})}{\|u_{n}\|v_{n}}\right]\chi_{[v_{n}\neq 0]}(x)v_{n}w dx\to 0.
\end{equation}
On the other hand, by Proposition \ref{Fatou}(ii), \eqref{weak} and Lebesgue dominated convergence theorem, we get
\begin{equation}\label{salio}
\int_{[v\neq 0]}\left[\frac{f(x, \|u_{n}\|v_{n})}{\|u_{n}\|v_{n}}\right]\chi_{[v_{n}\neq 0]}(x)v_{n}w dx\to \int_{\Omega}\eta(x)vw dx.
\end{equation}
It follows from \eqref{ahora} and \eqref{salio} that, passing to the limit as $n\to\infty$ in \eqref{cuenta}, we obtain
\begin{equation}\label{eigen}
\int_{\Omega}\nabla v\nabla w dx=\int_{\Omega}\eta(x)vw dx, \ \forall \ w\in H_{0}^{1}(\Omega).
\end{equation}

Now we have to consider two cases:

\medskip

$(i)$ If $\lambda_{m+k}(\eta)\neq 1$, for all $k\in\N$, it follows from \eqref{eigen} that $v=0$. But this is a contradiction. Therefore $\{u_{n}\}$ is bounded in $H_{0}^{1}(\Omega)$.

\medskip

$(ii)$ If $\lambda_{m+k}(\eta)= 1$, for some $k\in\N$, then \eqref{eigen} implies that $v = e_{m+k}$, where $e_{m+k}$ is some normalized eigenfunction associated to $\lambda_{m+k}(\eta)$. From \eqref{eigen}, it follows also that $\|v\|^{2}=\int_{\Omega}\eta(x)v^{2}dx$, i.e., $v\in \partial\mathcal{A}$. On the other hand,
$$
\int_{\Omega}\eta(x)v^{2}dx=\|v\|^{2}\leq\liminf_{n\to\infty}\|v_{n}\|^{2}=1.
$$
Suppose that 
\begin{equation}
\int_{\Omega}\eta(x)v^{2}dx<1.
\label{alphainfinity1}
\end{equation}
Since
\begin{equation}\label{equal3}
t_{v_{n}}=\|t_{v_{n}}v_{n}\|=\|u_{n}\|\to\infty,
\end{equation}
by arguing as in \eqref{salio} and \eqref{ahora}, we obtain
\begin{equation}\label{listo}
\int_{\Omega}\left[\frac{F(\|u_{n}\|v_{n})}{(\|u_{n}\|v_{n})^2}\right]v_{n}^{2}dx\to (1/2)\int_{\Omega}\eta(x)v^{2}dx.
\end{equation}
Thus, passing to the limit as $n\to\infty$ in the identity
$$
\Psi(v_{n})=\|u_{n}\|^{2}\left\{\frac{1}{2}-\int_{\Omega}\left[\frac{F(\|u_{n}\|v_{n})}{(\|u_{n}\|v_{n})^2}\right]v_{n}^{2}dx\right\}
$$
and using \eqref{listo} we conclude that $\Psi(v_{n})\to\infty$, a contradiction with \eqref{bound}. Consequently, 
\begin{equation}\label{boundary}
\|v\|^{2}=\int_{\Omega}\eta(x)v^{2}dx=1,
\end{equation}
showing that 
\begin{equation}\label{norm}
\|v_{n}\|\to \|v\|.
\end{equation}
By using \eqref{weak} and \eqref{norm}, we derive that $v_{n}\to v$ in $H_{0}^{1}(\Omega)$ with $v\in \partial\mathcal{S}_{\mathcal{A}}$ (see \eqref{boundary}). 

Invoking Proposition \ref{main1}, we conclude that
\begin{equation}\label{end}
c\geq \int_{[v\neq 0]}\beta(x) dx,
\end{equation}
with $v=e_{m+k}\in\partial\mathcal{S}_{\mathcal{A}}$, contradicting the fact that $c < \inf_{u\in \partial\mathcal{S}_{\mathcal{A}}}\int_{[u\neq 0]}\beta(x) dx$. Then conclude that $\{u_{n}\}$ is bounded. 

\medskip

Since $\{u_{n}\}$ is a bounded sequence, there exists $u\in H_{0}^{1}(\Omega)$ such that $u_{n}\rightharpoonup u$ in $H_{0}^{1}(\Omega)$, up to a subsequence. 

Then, what is left to prove is that $\|u_{n}\|\to \|u\|$. For this, it is sufficient to note that since $\{u_{n}\}$ is a $(PS)_{c}$ sequence, we have
$$
o_{n}(1)+\int_{\Omega}\nabla u_{n}\nabla u dx=\int_{\Omega}f(x, u_{n})u dx.
$$
Passing to the limit as $n\to\infty$ in the previous equality, we get
\begin{equation}\label{equal5}
\|u\|^{2}=\int_{\Omega}f(x, u)u dx.
\end{equation}
Then \eqref{equal5} and Lebesgue's convergence theorem imply that
$$
\|u_n\|^2 = \int_\Omega f(x,u_n)u_ndx = \int_\Omega f(x,u)udx + o_n(1) = \|u\|^2 + o_n(1).
$$
\end{proof}




\section{Signed ground-state solution}\label{sec:ground}

{\bf Proof of Thorem \ref{teo1}:}

Let $\{u_{n}\}\subset \mathcal{N}$ be such that $I(u_{n})\to c_{\mathcal{N}}$. Remember that $v_{n}:=u_{n}/\|u_{n}\|\in \mathcal{S}_{\mathcal{A}}$ (see Proposition \ref{proposition2}$(A_{3})$) and 
\begin{equation}\label{aca}
\Psi(v_{n})\to c_{\mathcal{N}}.
\end{equation}

We will show that $\{v_{n}\}$ is a $(PS)_{c_{\mathcal{N}}}$ sequence for the functional $\Psi$. For this, consider the map $\Upsilon:\overline{\mathcal{S}_{\mathcal{A}}}\to\R\cup\{\infty\}$ defined by 

\begin{equation}
\Upsilon(u)=\left \{ \begin{array}{ll}
\Psi(u) & \mbox{if $u\in \mathcal{S}_{\mathcal{A}}$,}\\
\int_{[u\neq 0]}\beta(x) dx & \mbox{if $u\in\partial\mathcal{S}_{\mathcal{A}}$.}
\end{array}\right.
\end{equation}

It follows from Lemma \ref{level} that $c_{\mathcal{N}}=\inf_{u\in \overline{\mathcal{S}_{\mathcal{A}}}}\Upsilon(u)$. Observe that $\overline{\mathcal{S}_{\mathcal{A}}}$ is a complete metric space with metric provided by the norm of $H_{0}^{1}(\Omega)$.

\medskip

{\bf Claim:} $\Upsilon$ is lower semi-continuous.\\

Let $\{u_{n}\}\subset \overline{\mathcal{S}_{\mathcal{A}}}$ and $u\in \overline{\mathcal{S}_{\mathcal{A}}}$ be such that $u_{n}\to u$ in $H_{0}^{1}(\Omega)$. If $u\in \mathcal{S}_{\mathcal{A}}$ then, for $n$ large enough, $\Upsilon(u_{n})=\Psi(u_{n})$ and
$$
\Upsilon(u_{n})=\Psi(u_{n})\to\Psi(u)=\Upsilon(u),
$$
since $\Psi$ is continuous. On the other hand, if $u\in \partial\mathcal{S}_{\mathcal{A}}$, we have two cases to consider. If there exists a subsequence $\{u_{n}\}\subset \mathcal{S}_{\mathcal{A}}$, then by Proposition \ref{main1}
$$
\Upsilon(u)=\int_{[u\neq 0]}\beta(x) dx\leq \liminf\limits_{n\to\infty}\Psi(u_{n})=\liminf\limits_{n\to\infty}\Upsilon(u_{n}).
$$
If there exists a subsequence $\{u_{n}\}\subset \partial\mathcal{S}_{\mathcal{A}}$ then by Proposition \ref{Fatou}(iii) with $v_{n}=v=\beta(x)$, we have
$$
\Upsilon(u)=\int_{[u\neq 0]}\beta(x) dx\leq \liminf\limits_{n\to\infty}\int_{[u_{n}\neq 0]}\beta(x) dx=\liminf\limits_{n\to\infty}\Upsilon(u_{n}).
$$
In any case, $\Upsilon$ is a is lower semi-continuous map and the claim is proven.

\medskip

Since $\Upsilon$ is bounded from below (see Lemma \ref{lemma3}), it follows from Theorem 1.1 in \cite{Ek} that for each $\varepsilon, \lambda>0$ 
small enough and $u\in \Upsilon^{-1}[c_\mathcal{N},c_\mathcal{N} + \varepsilon]$ there exists $v\in \overline{\mathcal{S}_{\mathcal{A}}}$ such that
\begin{equation}
c_{\mathcal{N}}\leq \Upsilon(v)\leq \Upsilon(u), \  \|u-v\|\leq \lambda \ \mbox{and} \ \Upsilon(w)>\Upsilon(v)-(\varepsilon/\lambda)\|v-w\|, \ \forall \ w\neq v.
\label{neweq}
\end{equation}
On the other hand, it follows from Lemma \ref{level} that
\begin{equation}\label{equali}
\Upsilon^{-1}[c_\mathcal{N},c_\mathcal{N} + \varepsilon]=\Psi^{-1}[c_\mathcal{N},c_\mathcal{N} + \varepsilon], v\in \mathcal{S}_{\mathcal{A}} \ \mbox{and} \ \Upsilon(v)=\Psi(v),
\end{equation} 
for $\varepsilon$ small enough. Up to a subsequence, it follows from \eqref{aca} that we can choose in (\ref{neweq}) $u=v_{n}$, $\varepsilon=1/n^{2}$ and $\lambda=1/n$, in order to get
$\widehat{v}_{n}\in \mathcal{S}_{\mathcal{A}}$, satisfying
\begin{equation}\label{PS1}
\Psi(\widehat{v}_{n})\to c_{\mathcal{N}}, \ \|v_{n}- \widehat{v}_{n}\|\to 0
\end{equation}
and
\begin{equation}\label{deriv}
\Upsilon(w)>\Psi(\widehat{v}_{n})-(1/n)\|\widehat{v}_{n}-w\|, \ \forall \ w\neq \widehat{v}_{n}.
\end{equation}

Let $\gamma_{n}: (-\delta_{n}, \delta_{n})\to \mathcal{S}_{\mathcal{A}}$ be a differentiable curve, with $\delta_{n}>0$ small enough, such that 
$\gamma_{n}(0)=\widehat{v}_{n}$ and $\gamma'_{n}(0)=z\in T_{\widehat{v}_{n}}(\mathcal{S}_{\mathcal{A}})$, where $T_{\widehat{v}_{n}}(\mathcal{S}_{\mathcal{A}})$ denotes the tangent space of $\mathcal{S}_\mathcal{A}$ at $\widehat{v}_n$.
Choosing $w=\gamma_{n}(t)$, it follows from $(\ref{deriv})$ that
\begin{equation}\label{deriv1}
-[\Psi(\gamma_{n}(t))-\Psi(\gamma_{n}(0))]<(1/n)\|\gamma_{n}(t)-\gamma_{n}(0)\|.
\end{equation}

By Mean Value Theorem, there exists $c \in (0,t)$ such that
\begin{equation}
 \|\gamma_{n}(t)-\gamma_{n}(0)\|
\leq \|  \gamma'_{n}(c)\|t. 
\label{deriv2}
\end{equation}
Thus, multiplying both sides of \eqref{deriv1} by $1/t$, passing to the limit as $t\rig 0$ and using \eqref{deriv2}, we get
$$
-\Psi'(\widehat{v}_{n})z\leq \frac{1}{n}\|z\|,
$$
where $z\in  T_{\widehat{v}_{n}}(\mathcal{S}_{\mathcal{A}})$ is arbitrary. By linearity, we have
$$
|\Psi'(\widehat{v}_{n})z|\leq \frac{1}{n}\|z\|.
$$
Therefore,
\begin{equation}\label{PS3}
\|\Psi'(\widehat{v}_{n})\|_{\ast}\to 0,
\end{equation}
as $n\to\infty$, and, by \eqref{PS1}, we conclude that $\{v_{n}\}$ is a $(PS)_{c_{\mathcal{N}}}$ sequence for $\Psi$. It follows from Lemma \ref{level} and Proposition \ref{main2} that there exists $v\in \mathcal{S}_{\mathcal{A}}$ such that, passing to a subsequence, $v_{n}\to v$ in $H_{0}^{1}(\Omega)$. Thus $\Psi'(v)=0$ and $\Psi(v)=c_{\mathcal{N}}$. Defining $u:=m(v)\in\mathcal{N}$ and using Proposition \ref{proposition3}$(iv)$, we conclude that $I'(u)=0$ and $I(u)=c_{\mathcal{N}}$.

To show that $u$ does not change sign, observe that if $u^\pm \neq 0$, then $u^{\pm}\in \mathcal{N}$ (to verify it just calculate $I'(u)u^\pm$). Then
\begin{equation}\label{sign}
c_{\mathcal{N}}=I(u)=I(u^{+})+I(u^{-})\geq 2c_{\mathcal{N}},
\end{equation}
which is a clear contradiction. Then it follows that either $u^+ = 0$ or $u^- = 0$ and then $u$ is a signed solution.

$\square$


\section{Multiplicity of solutions}\label{sec:multiplicity}

The main goal of this section is to prove Theorem \ref{teore2}. In its proof, we use the Krasnoselski's genus theory.  Thus, we start defining some preliminaries notations:
$$
\gamma_{j}:=\left\{B\in \mathcal{E}: B\subset \mathcal{S}_{\mathcal{A}} \ \mbox{and} \ \gamma(B)\geq j\right\},
$$
where 
$$
\mathcal{E}=\{B\subset H_{0}^{1}(\Omega)\backslash\{0\}: B \ \mbox{is closed and $B=-B$}\}
$$ 
and $\gamma:\mathcal{E}\to \mathds{Z}\cup \{\infty\}$ is the Krasnoselski's genus function, which is defined by

\begin{equation}
\gamma(B)=\left \{ \begin{array}{ll}
n:=\min\{ m\in\N: \ \mbox{there exists an odd map $\varphi\in C(B, \R^{m}\backslash\{0\})$}\},\\
\infty, \mbox{if there exists no map $\varphi\in C(B, \R^{m}\backslash\{0\})$},\\
0, \mbox{if $B=\emptyset$}.
\end{array}\right.
\end{equation}
It is important to note that, since $\mathcal{S}_{\mathcal{A}}=-\mathcal{S}_{\mathcal{A}}$, $\gamma_{j}$ is well defined.

In the sequel we will state some standard properties of the genus which will be used in this work. More information about this
subject can be found, for instance, in \cite{AR} or \cite{Kr}.

\begin{lemma}\label{genus}
Let $B$ and $C$ be sets in $\mathcal{E}$.

\begin{enumerate}

\item[$(i)$] If $x\neq 0$, then $\gamma(\{x\}\cup \{-x\})=1$;

\item[$(ii)$]  If there exists an odd map $\varphi\in C(B, C)$, then $\gamma(B)\leq \gamma(C)$. In particular, if $B\subset C$ then $\gamma(B)\leq \gamma(C)$.

\item[$(iii)$]  If there exists an odd homeomorphism $\varphi:B\rightarrow C$, then $\gamma(B)=\gamma(C)$. In particular, if $B$ is homeomorphic to the unit sphere in $\R^{n}$, then $\gamma(B)=n$.

\item[$(iv)$]  If $B$ is a compact set, then there exists a neighborhood $K\in \mathcal{E}$ of $B$ such that
$\gamma(B)=\gamma(K)$.

\item[$(v)$]  If $\gamma(C)<\infty$, then
$\gamma(\overline{B\backslash C})\geq \gamma(B)-\gamma(C)$.

\item[$(vi)$]  If $\gamma(A) \geq 2$, then $A$ has infinitely many points.
\end{enumerate}
\end{lemma}

From now on, we denote by $s_{m}$ the sum of the dimensions of all eigenspaces $V_{j}$ associated to eigenvalues $\lambda_{j}(\eta)$, where $1\leq j\leq m$. 

\begin{lemma}\label{nosso}
Suppose $(f_{2})$ holds. \textcolor{blue}{Then}
\begin{enumerate}
\item[$(i)$] $\gamma_{s_{m}}\neq \emptyset$;
\item[$(ii)$] $\gamma_{1}\supset\gamma_{2}\supset\ldots\supset\gamma_{s_{m}}$;
\item[$(iii)$] If $\varphi\in C(\mathcal{S}_{\mathcal{A}}, \mathcal{S}_{\mathcal{A}})$ is odd, then $\varphi(\gamma_j) \subset \gamma_j$, for all $1\leq j\leq s_{m}$;
\item[$(iv)$] If $B\in \gamma_{j}$ and $C\in \mathcal{E}$ with $\gamma(C)\leq s<j\leq s_{m}$, then $\overline{B\backslash C}\in \gamma_{j-s}$.
\end{enumerate}
\end{lemma}

\begin{proof}
$(i)$ Let $\mathcal{S}_{s_{m}}$ be the ($s_{m}$-dimensional) unit sphere of $V_{1}\oplus V_{2}\oplus\ldots\oplus V_{m}$. From $(f_{2})$, it is clear that $\mathcal{S}_{s_{m}}\subset \mathcal{S}_{\mathcal{A}}$. Moreover, from Lemma \ref{genus}$(iii)$, we have that $\gamma(\mathcal{S}_{s_{m}})=s_{m}$, showing that $\mathcal{S}_{s_{m}}\in \gamma_{s_{m}}$. $(ii)$ It is immediate. $(iii)$ It follows directly from Lemma \ref{genus}$(ii)$. $(iv)$ It is a consequence of Lemma \ref{genus}$(v)$.
\end{proof}

Now, for each $1\leq j\leq s_{m}$, we define the following minimax levels
\begin{equation}
c_{j}=\inf_{B\in \gamma_{j}}\sup_{u\in B}\Psi(u).
\end{equation}

\begin{lemma}\label{minimax}
Suppose $(f_{1})-(f_{2})$ and \eqref{finito} to hold. Then,
$$
0< c_{\mathcal{N}}=c_{1}\leq c_{2}\leq \ldots\leq c_{s_{m}}< \inf_{u\in\partial\mathcal{S}_{\mathcal{A}}}\int_{[u\neq 0]}\beta(x) dx.
$$


\end{lemma}

\begin{proof}
$(i)$ The First inequality follows from Lemma \ref{lemma3} (see also Remark \ref{rem2}). The equality $c_{\mathcal{N}}=c_{1}$ can be easily obtained from Lemma \ref{genus}$(i)$ and the definition of $c_{1}$. On the other hand, the monotonicity of the levels $c_{j}$ is a consequence of Lemma \ref{nosso}$(ii)$. To prove last inequality, observe that by the proof of Lemma \ref{level}, we have
$$
\Psi(u)\leq [1/2\lambda_{1}(\eta-\alpha)]t_{u}^{2}, \ \forall \ u\in \mathcal{S}_{\mathcal{A}}.
$$
Thus, 
\begin{equation}\label{zzz}
c_{s_{m}}\leq\max_{u\in \mathcal{S}_{s_{m}}}\Psi(u)\leq [1/2\lambda_{1}(\eta-\alpha)]\tau_{m}^{2},
\end{equation}
where, by items $(A_{1})$ and $(A_{2})$ of Proposition \ref{proposition2}
\begin{equation}\label{numer}
0<\tau_{m}:=\max_{u\in \mathcal{S}_{s_{m}}}t_{u}<\infty.
\end{equation}
The result follows now from Lemma \ref{limita}, \eqref{finito} and \eqref{zzz}.

%
\end{proof}

Next proposition is crucial to ensure the multiplicity of solutions.

\begin{proposition}\label{mult}
Suppose that $f$ satisfies $(f_{1})-(f_{2})$ and \eqref{finito}. If $c_{j}=\ldots=c_{j+p}\equiv c$, $j+p\leq s_{m}$, then $\gamma(K_{c})\geq p+1$, where $K_{c}:=\{v\in \mathcal{S}_{\mathcal{A}}: \Psi(v)=c\ \ \mbox{and} \ \Psi'(v)=0\}$.
\end{proposition}

\begin{proof}
Suppose that $\gamma(K_{c})\leq p$. It follows from Proposition \ref{main2} and Lemma \ref{minimax}$(i)$ that $K_{c}$ is a compact set. Thus, by Lemma \ref{genus}$(iv)$, there exists a compact neighborhood $K$ (in $H_{0}^{1}(\Omega)$) of $K_{c}$ such that $\gamma(K)\leq p$. Defining $M:=K\cap \mathcal{S}_{\mathcal{A}}$, we derive from Lemma \ref{genus}$(ii)$ that $\gamma(M)\leq p$. Despite the noncompleteness of $\mathcal{S}_{\mathcal{A}}$ we can use Theorem 3.11 in \cite{Str} (see also Remark 3.12 in \cite{Str}) to ensure the existence of an odd homeomorphisms family $\eta(., t)$ of $\mathcal{S}_{\mathcal{A}}$ such that, for each $u\in \mathcal{S}_{\mathcal{A}}$, 
\begin{equation}\label{iguald}
\eta(u, 0)=u
\end{equation} 
and
\begin{equation}\label{traj}
t\mapsto\Psi(\eta(u, t)) \ \mbox{is non-increasing}.
\end{equation}


Observe that, although $\mathcal{S}_{\mathcal{A}}$ is non-complete, from Proposition \ref{main1}, Lemma \ref{minimax} and \eqref{traj}, for $\varepsilon>0$ small enough, map $\eta(u, .)$ is well defined in $[0,\infty)$, for each $u\in \Psi_{c_{s_{m}}+\varepsilon}=\{u\in \mathcal{S}_{\mathcal{A}}: \Psi(u)\leq c_{s_{m}}+\varepsilon\}$. Indeed, suppose by contradiction that $\eta(u, t_{0})\in \partial\mathcal{S}_{\mathcal{A}}$ for some $u\in \Psi_{c_{s_{m}}+\varepsilon}$ and $t_{0}>0$, where $\varepsilon\in (0, \inf_{u\in \partial\mathcal{S}_{\mathcal{A}}}\int_{[u\neq 0]}\beta(x) dx- c_{s_{m}})$. Then, by \eqref{iguald}, Proposition \ref{main1} and Lemma \ref{minimax}
$$
\Psi(\eta(u, 0))=\Psi(u)\leq c_{s_{m}}+\varepsilon<\int_{[\eta(u, t_{0})\neq 0]}\beta(x) dx\leq \liminf\limits_{t\to t_{0}}\Psi(\eta(u, t)).
$$
Thus, there exists $0<t_{\ast}<t_{0}$, such that $\eta(u, t_{\ast})\in \mathcal{S}_{\mathcal{A}}$ and
$$
\Psi(\eta(u, 0))<\Psi(\eta(u, t_{\ast})),
$$
what contradicts \eqref{traj}.

Therefore, it makes sense the  third claim of Theorem 3.11 in \cite{Str}, namely,
\begin{equation}\label{def}
\eta(\Psi_{c+\varepsilon}\backslash M, 1)\subset \Psi_{c-\varepsilon}.
\end{equation}
Let us choose $B\in \gamma_{j+p}$ such that $\sup_{B}\Psi\leq c+\varepsilon$. From Lemma \ref{nosso}$(iv)$, $\overline{B\backslash M}\in\gamma_{j}$. It follows again from Lemma \ref{nosso}$(iii)$ that $\eta(\overline{B\backslash M}, 1)\in \gamma_{j}$. Therefore, by \eqref{def} and the definition of c, we have
$$
c\leq \sup_{\eta(\overline{B\backslash M}, 1)}\Psi\leq c-\varepsilon,
$$
that is a contradiction. Then $\gamma(K_{c})\geq p+1$.
\end{proof}

We are now ready to prove the following multiplicity result.

\vspace{0.5cm}

\noindent {\bf Proof of Theorem \ref{teore2}:}

First of all, note that the levels $0<c_{j}<\infty$ are critical levels of $\Psi$. In fact, suppose by contradiction that $c_{j}$ is regular for some $j$. Invoking Theorem 3.11 in \cite{Str}, with $\beta=c_{j}$, $\overline{\varepsilon}=1$, $N=\emptyset$, there exist $\varepsilon>0$ and a family of odd homeomorphisms $\eta(., t)$ satisfying the properties of such theorem. Choosing $B\in \gamma_{j}$ such that $\sup_{B}\Psi<c_{j}+\varepsilon$ and arguing as in the proof of Proposition \ref{mult} we get a contradiction. 

Finally, if levels $c_{j}$, $1\leq j\leq s_{m}$, are different from each other, it follows from Proposition \ref{proposition3}$(iv)$ that the result follows. On the other hand, if $c_{j}=c_{j+1}\equiv c$ for some $1\leq j\leq s_{m}$, it follows from Proposition \ref{def} that $\gamma(K_{c})\geq 2$. Combining the last inequality with Lemma \ref{genus}$(vi)$ and Proposition \ref{proposition3}$(iv)$, we conclude that \eqref{P} has infinitely many pairs of nontrivial solutions.
$\square$

\end{document}